\documentclass[a4paper]{amsart}

\usepackage{amssymb}
\usepackage{color}
\usepackage[pdfpagelabels]{hyperref}												
\usepackage{enumerate}

\newcommand{\chr}{\operatorname{chr}}

\newtheorem{theorem}{Theorem}[section]
\newtheorem{proposition}[theorem]{Proposition}

\theoremstyle{definition}
\newtheorem{definition}[theorem]{Definition}

\theoremstyle{remark}
\newtheorem{remark}[theorem]{Remark}

\begin{document}

\title{Colouring simplicial complexes via the Lechuga-Murillo's model}

\author{David ~M\'endez}
\address[D.\ M\'endez]{School of Mathematical Sciences, University of Southampton, SO17 1BJ, United Kingdom}
\email{D.Mendez-Martinez@soton.ac.uk}
\thanks{The author was partially supported by Ministerio de Econom\'ia y Competitividad (Spain) grants MTM2016-79661-P and MTM2016-78647-P}

\begin{abstract}
	L.\ Lechuga and A.\ Murillo showed that a non-oriented, simple, connected, finite graph $G$ is $k$-colourable if and only if a certain pure Sullivan algebra associated to $G$ and $k$ is not elliptic. In this paper, we extend this result to simplicial complexes by means of several notions of colourings of these objects.
\end{abstract}

\maketitle

\section{Introduction}\label{sec:intro}

Graph Theory and Rational Homotopy Theory were first related by L.\ Lechuga and A.\ Murillo in a celebrated paper \cite{LecMur00} (see also \cite{LecMur01}) where they show that a non-oriented, simple, connected, finite graph can be $k$-coloured, $k\ge 2$, if and only if a certain Sullivan algebra associated to the graph is not elliptic. They also provide a link between Rational Homotopy Theory and algorithmic complexity by proving that the problem of graph colourability can be reduced in polynomial time to the problem of determining the ellipticity of a certain Sullivan algebra. Hence, since the former is an \emph{NP}-complete problem, the latter is an \emph{NP}-hard problem.

This interplay between Graph Theory and Rational Homotopy Theory has been proven fruitful: recently, C.\ Costoya and A.\ Viruel were able to use this interaction to solve a question of realisability of groups \cite{CosVir14, CosVir14b}, and applications of these results to further problems were subsequently found, \cite{CosMenVir18,CosMenVir19}.

The aim of this work is to extend the result of Lechuga and Murillo from graphs to (finite) simplicial complexes by considering eleven notions of colourability for these objects, many of which can be found in the literature. We refer to these colourings as $\mathfrak{C}_i$-colouring,  for $i=1,2, \dots, 11$ (see Definitions \ref{definition:hyperedgeColourings}, \ref{definition:simplicialColourings}, \ref{def:colorDMN}, and \ref{definition:maximalminimalcolourings}), and prove the following two results:

\begin{theorem}\label{theorem:all} 
	For any $k\ge 2$, any $i=1,2,\dots,11$, and any connected simplicial complex $X$, which is assumed to be strongly connected and homogeneous for $i = 8,9,10,11$, there exists a pure Sullivan algebra ${\mathcal M}^i_k (X)$ which is not elliptic if and only if $X$ is $\mathfrak C_i$-$k$-colourable. 
\end{theorem}

\begin{theorem}\label{theorem:NP-hardness}
	For $i\in\{1,7,8,9,10,11\}$ and $k\ge 3$, or for $i\in\{4,5,6\}$ and $k\ge 4$, determining if a connected simplicial complex is $\mathfrak{C}_i$-$k$-colourable is an NP-hard problem.
\end{theorem}

We point out that closely related problems have been studied in \cite{CosVir15,DobMolNot10,MolNor16}.

As for the necessary background, we assume that the reader is familiar with basics of algorithmic complexity and Rational Homotopy Theory, for which \cite{Wil02} and \cite{FelHalTho01} are, respectively, excellent references. In particular, concerning algorithmic complexity we will use that the problems of total-$k$-colourability, $k\ge 4$, edge-$k$-colourability, $k\ge 3$, and $k$-colourability, $k\ge 3$, are \emph{NP}-complete, \cite{Hol81, LevGal83,San89}. 

Regarding Rational Homotopy Theory, we just recall that a (simply connected) \emph{Sullivan algebra}, denoted $(\Lambda W,d)$, is a commutative differential graded algebra, which is free as an algebra generated by the (simply connected) graded rational vector space $W$, and where the differential $d$ is decomposable. A Sullivan algebra is \emph{elliptic} if both $W$ and $H^*(\Lambda W,d)$ are finite dimensional, and \emph{pure} if $d W^{\text{even}}=0$ and $d W^{\text{odd}}\subset \Lambda W^{\text{even}}$.

We now recall the fundamental construction in \cite{LecMur00} associated to any $k\ge 2$ and any non-oriented, simple, connected, finite graph $G = (V,E)$, where $V$ and $E$ respectively denote the sets of vertices and edges of $G$. Consider the pure Sullivan algebra $S_k(G)  =(\Lambda W_{G,k},d)$ where
\begin{align*}
	W_{G,k}^{\text{even}} &= \langle x_v \mid v\in V \rangle , &|x_v|&=2, &d(x_v)&=0,  \\
	W_{G,k}^{\text{odd}} &= \langle y_{(u, v)} \mid (u, v)\in E\rangle,  &|y_{(u, v)}| &= 2k-3,  &d(y_{(u, v)}) &= \Sigma_{l=1}^k x_u^{k-l}x_v^{l-1}.
\end{align*}
For this construction, the following holds:

\begin{theorem}\emph{(\cite[Theorem 3]{LecMur00})}\label{theorem:lechugamurillo}
The graph $G$ is $k$-colourable if and only if the Sullivan algebra $S_k(G)$ is not elliptic.
\end{theorem}

To relate this result with algorithmic complexity it is convenient to keep in mind that a graph $G = (V,E)$ is usually encoded by its \emph{adjacency matrix} $A=(a_{ij})_{i,j\in V}$ in which $a_{ij}=1$ if $(i, j)\in E$ and $a_{ij}=0$ otherwise. In binary, the codification of this matrix has length $\log_2 n+n^2$, where $n$ is the number of vertices of $G$.

Throughout this paper, every considered simplicial complex $X$ is assumed to be finite. The dimension of a simplex $\sigma\in X$, denoted $\dim\sigma$, is its cardinality minus one. The \emph{dimension} of $X$, denoted $\dim X$, is the dimension of any of its largest simplices. Given $s\ge 0$, we denote the set of simplices of $X$ of dimension $s$ by $X^s$. In particular, $X^0$ is the set of vertices of $X$, which is often denoted by $V$. The \emph{$s$-skeleton} of $X$ is the subsimplicial complex of X spanned by $X^s$, and we denote it by $X^{(s)}$. Note that $X^{(1)}$ is trivially identified to a non-oriented, simple graph, and we say that $X$ is \emph{connected} if $X^{(1)}$ is a connected graph.

\section{Models for colourings of connected simplicial complexes}

In the spirit of Theorem \ref{theorem:lechugamurillo}, we will associate to finite, connected simplicial complexes precise pure Sullivan algebras whose ellipticity encode different notions of colouring of simplicial complexes.

\subsection{Colourings arising from hypergraphs}\label{sub: hypergraphcolour}

Recall that a \emph{hypergraph} is a pair $H = (V,E)$ formed by a non-empty set of vertices $V$ and a set of hyperedges $E$, each of them being a non-empty subset of $V$. Two vertices are adjacent if they belong to a common hyperedge. An hyperedge $e$ is incident to a vertex $v$ if $v\in e$. Two hyperedges $e$ and $e'$ are adjacent if $e\cap e'\ne\emptyset$. The hypergraph $H$ is connected if given any two vertices $u,v\in V$ there is a sequence of hyperedges $e_1,e_2,\dots,e_n$ such that $u\in e_1$, $v\in e_n$ and $e_i$ is adjacent to $e_{i+1}$, for $i=1,2,\dots,n-1$.

A \emph{vertex $k$-colouring} of a hypergraph $H = (V,E)$, see \cite[\textsection3.1]{Bre13}, is a map $\varphi\colon V\to \{1,2,\dots,k\}$ such that for any hyperedge $e$ of more than one vertex $|\varphi(e)| > 1$. Namely, at least two vertices of $e$ have different colours. Moreover, if for any $e\in E$ and any two different vertices $u, v\in e$ we have that $\varphi(u)\ne\varphi(v)$, we say that $\varphi$ is a \emph{strong} vertex $k$-colouring.

On the other hand, \cite[\textsection3.2.5]{Bre13} a \emph{hyperedge colouring} for $H$ is a map $\varphi\colon E\to \{1,2,\dots,k\}$ such that $\varphi(e)\ne\varphi(e')$ for any pair of different but adjoint hyperedges $e$ and $e'$.

Finally, \cite{Cow97}, a \emph{total colouring} of $H$ is a map $\varphi \colon V\cup E \to \{1,2,\dots,k\}$ such that any pair formed by either two adjacent vertices, two adjacent hyperedges or an hyperedge and any of its incident vertices have different images through $\varphi$.

Trivially, a simplicial complex $X$ can be regarded as a hypergraph $H = (V,E)$ where $V = X^0$ and $E = X$. Hence, the above notions of colourability automatically translate to the following definition. Note that a vertex $k$-colouring of a simplicial complex is always a strong vertex $k$-colouring.

\begin{definition}\label{definition:hyperedgeColourings}
	Let $X$ be a simplicial complex.
	\begin{enumerate}
		\item A \emph{vertex $k$-colouring} of $X$ ($\mathfrak{C}_1$-$k$-colouring) is a map $\varphi\colon V\to\{1,2,\dots,k\}$ such that if $\sigma\in X$ and $u,v\in\sigma$, $u\ne v$, then $\varphi(u)\ne\varphi(v)$.

		\item A \emph{face $k$-colouring} of $X$ ($\mathfrak{C}_2$-$k$-colouring) is a map $\varphi\colon X\to\{1,2,\dots,k\}$ such that $\varphi(\sigma)\ne\varphi(\tau)$ whenever $\sigma\ne\tau$, $\sigma\cap\tau\ne\emptyset$.

		\item A \emph{total $k$-colouring} of $X$ ($\mathfrak{C}_3$-$k$-colouring) is a map $\varphi\colon X\to\{1,2,\dots,k\}$ such that $\varphi(u)\ne\varphi(v)$ for any $u,v\in V$ with $\{u,v\}\in X$, and $\varphi(\sigma)\ne\varphi(\tau)$ for any pair of different simplices $\sigma,\tau$ with non-empty intersection.
	\end{enumerate}
\end{definition}

Note that a total $k$-colouring yields both a vertex $k$-colouring and a face $k$-colouring. We prove:

\begin{proposition}\label{proposition:hyperedgeColourings}
	For any simplicial complex $X$ and any $i =1,2,3$, there is a pure Sullivan algebra $\mathcal{M}_k^i(X)$ which is not elliptic if and only if $X$ is $\mathfrak{C}_i$-$k$-colourable.
\end{proposition}

\begin{proof}
	Associated to $X$ consider $G_1 = X^{(1)}$ the graph given by its $1$-skeleton. On the other hand, let $G_2$ be the graph whose vertex set is the set of simplices of $X$ and whose edges are pairs of distinct simplices with a common face. Finally, let $G_3$ be the graph whose vertex set is again the set of simplices of $X$ and whose edges are also pairs of distinct simplices with non-empty intersection, together with pairs of vertices giving raise to a $1$-simplex. Observe that $G_1$, $G_2$ and $G_3$ are respectively the \emph{$2$-section graph}, \emph{intersection graph} and \emph{total graph} of the hypergraph given by $X$ (see \cite{Bre13,Cow97}).

	It is then clear from Definition \ref{definition:hyperedgeColourings} that a $\mathfrak{C}_i$-$k$-colouring of $X$ is precisely a $k$-colouring of $G_i$, $i = 1,2,3$. Furthermore, the graphs $G_1$, $G_2$ and $G_3$ are connected as a consequence of $X$ being connected. To finish, define $\mathcal{M}_k^i(X) = S_k(G_i)$ and apply Theorem \ref{theorem:lechugamurillo}.
\end{proof}

\subsection{Colourings of simplicial complexes}\label{sub: simplexcolour}

The colourings in \textsection\ref{sub: hypergraphcolour} are originally defined for hypergraphs, thus they do not take consideration of the additional structure of simplicial complexes. For that reason, we introduce the following:

\begin{definition}\label{definition:ascendingDescending}
	Let $X$ be a simplicial complex.
	\begin{enumerate}
		\item An \emph{ascending $k$-colouring of $X$ in dim $r$} is a map $\varphi\colon X^r \to \{1,2,\dots,k\}$ such that if $\sigma,\tau\in X^r$, $\sigma\cup\tau\in X^{r+1}$, then $\varphi(\sigma)\ne\varphi(\tau)$. 

		\item A \emph{descending $k$-colouring of $X$ in dim $r$} is a map $\varphi\colon X^r \to \{1,2,\dots,k\}$ such that if $\sigma,\tau\in X^r$, $\sigma\cap\tau\in X^{r-1}$, then $\varphi(\sigma)\ne\varphi(\tau)$.
	\end{enumerate}

	We denote the respective chromatic numbers by $\chi_r(X)$ and $\chi'_r(X)$.
\end{definition}

An ascending $k$-colouring of $X$ in dim $r$ is a colouring of the graph 
\begin{equation}\label{eq:asc}
	G_r(X)=\big(X^r,\{(\sigma,\tau)\mid \sigma\cup\tau \in X^{r+1}\}\big),
\end{equation} whereas a descending $k$-colouring of $X$ in dim $r$ is a colouring of 
\begin{equation}\label{eq:desc}
	G'_r(X)=\big(X^r,\{(\sigma,\tau)\mid \sigma\cap\tau\in X^{r-1}\}\big),
\end{equation} 
called the \emph{$r$-th exchange graph} of $X$ (see \cite{Gru69}). However, Theorem \ref{theorem:lechugamurillo} cannot be used to model the colourings in Definition \ref{sub: hypergraphcolour} using these graphs, as they may not be connected. We treat this issue in \textsection\ref{sub: strongly}. 

Instead, in this section we use the ascending and descending colourings to introduce new colourings which we can model in the spirit of Proposition \ref{proposition:hyperedgeColourings}.

\begin{definition}\label{definition:simplicialColourings}
Let $X$ be a simplicial complex.

	\begin{enumerate}
		\item A \emph{complete ascending $k$-colouring} of $X$ ($\mathfrak{C}_4$-$k$-colouring) is a map $\varphi\colon X\to\{1,2,\dots,k\}$ such that, for any $r,s\in\{0,1,\dots,\dim X\}$, if $\sigma,\tau\in X^r$, $\sigma\cup\tau\in X^{r+1}$, or if $\sigma\in X^r$, $\tau\in X^s$, $r\ne s$, then $\varphi(\sigma)\ne\varphi(\tau)$.

		\item A \emph{complete descending $k$-colouring of $X$}  ($\mathfrak{C}_5$-$k$-colouring) is a map $\varphi\colon X\to \{1,2,\dots,k\}$ such that, for any $r,s\in \{0,1,\dots,\dim X\}$, if $\sigma,\tau\in X^r, \sigma\cap\tau\in X^{r-1}$, or if $\sigma\in X^r, \tau\in X^s, r\ne s$, then $\varphi(\sigma)\ne\varphi(\tau)$.

		\item A map $\varphi\colon X \to \{1,2,\dots,k\}$ is a \emph{full $k$-colouring of $X$}  ($\mathfrak{C}_6$-$k$-colouring) if for $\sigma,\tau\in X$ such that $\sigma\subset \tau$, or $\sigma,\tau\in X^0$, $\sigma\cup\tau\in X^1$, or for $1\le r\le\dim X$, $\sigma,\tau\in X^r$, $\sigma\cap\tau\in X^{r-1}$, we have that $\varphi(\sigma)\ne\varphi(\tau)$.
	\end{enumerate}
\end{definition}

Let $G_1 = (V_1, E_1)$ and $G_2 = (V_2, G_2)$ be two graphs. Recall that the sum of $G_1$ and $G_2$ is a graph $G = G_1+G_2$ with vertex set $V_1\sqcup V_2$ and edges $E_1\cup E_2\cup\{(u,v)\mid u\in V_1, v\in V_2\}$. The sum of any two graphs is connected. Also recall that the union of $G_1$ and $G_2$ is the graph $G_1\cup G_2$ with vertex set $V_1\cup V_2$ and edges $E_1\cup E_2$. 

\begin{proposition}\label{proposition:simplicialColourings}
	For any simplicial complex $X$ and any $i=4,5,6$, there is a pure Sullivan algebra $\mathcal{M}_k^i(X)$ which is not elliptic if and only if $X$ is $\mathfrak{C}_i$-$k$-colourable.
\end{proposition}

\begin{proof}
	First, note that a complete ascending (resp.\ descending) $k$-colouring of $X$ is an ascending (resp.\ descending) $k$-colouring of $X$ in dim $r$ when restricted to $X^r$. Furthermore, simplices of different dimensions receive different colours. It becomes clear that if we define
	\begin{align*}
		G_4 & = G_0(X) + G_1(X) + \dots + G_{\dim X}(X), \\
		G_5 & = G'_0(X) + G'_1(X) + \dots + G'_{\dim X}(X), 
	\end{align*}
	$X$ admits a complete ascending (resp.\ descending) $k$-colouring if and only if the connected graph $G_4$ (resp.\ $G_5$) is $k$-colourable.

	Regarding the full $k$-colouring, let $I$ denote the strict inclusion graph of $X$, that is, a graph with vertex set $X$ and where $(\sigma,\tau)$ is an edge if and only if either $\sigma\subset\tau$ or $\tau\subset\sigma$. Define a graph
		\[G_6 = I \cup \big(G_0(X)\sqcup G_1'(X)\sqcup \dots \sqcup G'_{\dim X}(X)\big).\]
	Then $G_6$ is connected since $I$ is so. Furthermore, $X$ is full-$k$-colourable if and only if $G_6$ is $k$-colourable. To finish, define $\mathcal{M}_k^i(X) = S_k(G_i)$, $i=4,5,6$, and apply Theorem \ref{theorem:lechugamurillo}.
\end{proof}

We model one last colouring in this section. In \cite{DobMolNot10} the authors introduce the following, more relaxed definition of vertex colouring:

\begin{definition}\label{def:colorDMN} 
	Let $k, s \geq 1$ and let $X$ be a simplicial complex. A $(k, s)$-colouring of $X$ ($\mathfrak{C}_7$-$(k,s)$-colouring) is a map $f: V \rightarrow \{1,2,\dots,k\}$ such that, for every $\sigma \in X $ and for all $1\le t \le k$, $| \sigma \cap f^{-1}(t) | \leq s$. Let $\chr^s(X)$ denote the least integer $k$ such that $X$ is $(k,s)$-colourable.
\end{definition}

A Sullivan algebra whose ellipticity codifies the $(k,s)$-colourability of a simplicial complex had already been obtained in \cite{CosVir15}. However, we can use the work in \cite{MolNor16} to provide a different construction of one such algebra:

\begin{proposition}\label{proposition:kscolouring} 
	For any simplicial complex $X$ there exists a pure Sullivan algebra $\mathcal{M}_{k,s}^7(X)$ which is not elliptic if and only if $X$ is $\mathfrak{C}_7$-$(k,s)$-colourable.
\end{proposition}

\begin{proof}
	In \cite[Theorem 2]{MolNor16} the authors show that 
	\begin{equation*}
    	\chr^s (X) = \min_{P \in \mathrm{BCP}^s(X)} \chr^1 \big(G_0(P)\big),
	\end{equation*}
	where $\mathrm{BCP}^s(X)$ is a set of partitions of the vertex set of $X$ and $G_0(P)$ is a $1$-dimensional simplicial complex associated to one such partition $P$, see \cite[Definition 3]{MolNor16}. It quickly follows that when regarding $G_0(P)$ as a graph,  $\chr^1 \big(G_0(P)\big)=\chi\big(G_0(P)\big)$. Furthermore, $G_0(P)$ is connected for every $P\in \text{BCP}^s(X)$. Define
		\[\mathcal{M}_{k,s}^7(X) = \bigotimes_{P \in \mathrm{BCP}^s(X)}   S_k\big(G_0(P)\big).\]
	Let us show that $\mathcal{M}_{k,s}^{7}(X)$ is the desired algebra.

	Recall that the tensor product of Sullivan algebras is not elliptic if and only if at least one of the factors is not elliptic. Therefore, if $\mathcal{M}_{k,s}^7(X)$ is not elliptic, there exists $P\in\mathrm{BCP}^s(X)$ such that $S_k\big(G_0(P)\big)$ is not elliptic. Then by Theorem \ref{theorem:lechugamurillo} $G_0(P)$ is $k$-colourable, so $\chi\big(G_0(P)\big)=\chr^1\big(G_0(P)\big)\le k$, thus $X$ is $(k,s)$-colourable. Reciprocally, if $\mathcal{M}_{k,s}^7(X)$ is elliptic, then $S_k\big(G_0(P)\big)$ is elliptic for every $P\in \textrm{BCP}^s(X)$. Therefore, $G_0(P)$ is not $k$-colourable, meaning that $\chi\big(G_0(P)\big)=\chr^1\big(G_0(P)\big) > k$, for every $P\in \textrm{BCP}^s$. Therefore, $X$ is not $(k,s)$-colourable. 
\end{proof}

\section{Models for colourings of strongly connected homogeneous simplicial complexes}\label{sub: strongly}

As mentioned in \textsection\ref{sub: simplexcolour},  the colourings in Definition \ref{definition:ascendingDescending} cannot be immediately modelled since the graphs that encode them, $G_r(X)$ (see \eqref{eq:asc}) and and $G_r'(X)$ (see \eqref{eq:desc}), are not necessarily connected. In this section we further restrict the class of simplicial complexes that we are considering as to be able to model these colourings. 

Recall that a simplicial complex $X$ of dimension $\dim X = n$ is \emph{strongly connected} if for any two $n$-dimensional simplices $\sigma$, $\tau$ there exist $\{\sigma_0 = \sigma, \sigma_1,\dots,\sigma_k = \tau\}\subset X^n$ such that $\sigma_{i-1}\cap\sigma_i\in X^{n-1}$, for $i = 1,2,\dots,k$. Equivalently, $X$ is strongly connected if and only if $G'_{\dim(X)}$ is connected. On the other hand, $X$ is \emph{homogeneous} if every vertex is contained in an $n$-dimensional simplex. Then, if $X$ is homogeneous and strongly connected, so is $X^{(k)}$, for $0\le k \le n$. Therefore:

\begin{proposition}\label{prop:grconexos}
	For any $n$-dimensional strongly connected homogeneous simplicial complex $X$,  $G_r(X)$ and $G'_s(X)$ are connected, for $0\le r < n$ and $0 < s \le n$.
\end{proposition}

\begin{proof}
	The connectivity of $G'_s(X)$, for $0< s \le n$ is an immediate consequence of the strong connectivity of $X^{(s)}$. Let us prove the connectivity of $G_r(X)$, $0\le r < n$. Take $\sigma,\tau\in X^r$. Since $X$ is homogeneous, we can find $\bar{\sigma},\bar{\tau}\in X^{r+1}$ such that $\sigma\subset\bar{\sigma}$ and $\tau\subset\bar{\tau}$. Then, since $X^{(r+1)}$ is strongly connected, we can find $\{\bar{\sigma}_0 = \bar{\sigma}, \bar{\sigma}_1,\dots,\bar{\sigma}_k = \bar{\tau}\}\subset X^{r+1}$ such that $\sigma_i = \bar{\sigma}_{i-1}\cap\bar{\sigma}_{i}\in X^r$, $i = 1,2,\dots,k$. It is now immediate to check that $\sigma \sigma_1 \dots \sigma_{k} \tau$ is a path in $G_r(X)$ joining $\sigma$ and $\tau$.
\end{proof}

 An immediate application of Theorem \ref{theorem:lechugamurillo} yields the following result:

\begin{proposition}
 	For any $n$-dimensional strongly connected homogeneous simplicial complex $X$ and for $0\le r <n$ (resp.\ for $0<s\le n$), there exists a Sullivan algebra $\mathcal{M}_k (X,r)$ (resp.\ $\mathcal{M}'_k(X,s)$) which is not elliptic if and only if $X$ admits an ascending $k$-colouring in dim $r$ (resp.\ a descending $k$-colouring in dim $s$).
\end{proposition}

We now introduce the last collection of colourings.
\begin{definition}\label{definition:maximalminimalcolourings}
	We say that a map $\varphi\colon X\to \{1,2,\dots,k\}$ is:
	\begin{itemize}
		\item a \emph{maximal ascending $k$-colouring} ($\mathfrak{C}_8$-$k$-colouring) if for every $0\le r \le \dim X$ the restriction $\varphi_{|X^r}$ is an ascending $k$-colouring in dim $r$ for $X$.
		\item a \emph{maximal descending $k$-colouring} ($\mathfrak{C}_{9}$-$k$-colouring) if for every $0\le s \le \dim X$ the restriction $\varphi_{|X^s}$ is a descending $k$-colouring in dim $s$ for $X$.
		\item a \emph{minimal ascending $k$-colouring} ($\mathfrak{C}_{10}$-$k$-colouring) if there exists $0\le r < \dim X$ such that $\varphi_{|X^r}$ is an ascending $k$-colouring in dim $r$ for $X$.
		\item a \emph{minimal descending $k$-colouring} ($\mathfrak{C}_{11}$-$k$-colouring) if there exists $0 < s \le \dim X$ such that $\varphi_{|X^s}$ is a descending $k$-colouring in dim $s$ for $X$.
	\end{itemize}
	The respective chromatic numbers are denoted $\chi_{\max}(X)$, $\chi'_{\max}(X)$, $\chi_{\min}(X)$ and $\chi'_{\min}(X)$.
\end{definition}

Let $G_1 = (V_1, E_1)$ and $G_2 = (V_2, G_2)$ be two graphs. The cartesian product $G_1\square G_2$ is a graph with vertex set $V_1\times V_2$ and edge set $\{\big((u_1,u_2),(v_1,v_2)\big)\mid\text{$u_1 = v_1$ and $(u_2,v_2)\in E_2$ or $(u_1,v_1)\in E_1$ and $u_2 = v_2$}\}$. Note that $\chi(G_1\square G_2)=\max\big\{\chi(G_1),\chi(G_2)\big\}$ (see \cite[Theorem 26.1]{HamImrKla11}). Furthermore, the cartesian product of connected graphs is connected (\cite[Corollary 5.3]{HamImrKla11}). Then:

\begin{proposition}\label{proposition:maximalMinimal}
	For any simplicial complex $X$ and any $i = 8,9,10,11$, there is a pure Sullivan algebra $\mathcal{M}_k^i(X)$ which is not elliptic if and only if $X$ is $\mathfrak{C}_i$-$k$-colourable.
\end{proposition}

\begin{proof}
	Note that any map $X^{\dim(X)}\to\{1,2,\dots,k\}$ (resp.\ $X^0\to\{1,2,\dots,k\}$) is an ascending colouring in dimension $\dim(X)$ (resp.\ a descending colouring in dimension $0$). Then, $\chi_{\dim(X)}(X) = \chi'_0(X)=1$. It follows immediately from Definition \ref{definition:maximalminimalcolourings} that $\chi_{\max}(X) = \max_{0 \le r < \dim X} \{\chi_r(X)\}$ and that $\chi'_{\max}(X) = \max_{0 < s\le\dim X}\{\chi'_s(X)\}$. Consider the graphs
	\[G_8 = G_0(X)\square G_1(X)\square\cdots\square G_{n-1}(X), \ G_9 = G'_1(X)\square G'_2(X)\square\cdots\square G'_n(X).\] 
	Then, since $\chi\big(G_r(X)\big) = \chi_r(X)$ and $\chi\big(G'_s(X)\big) =\chi'_s(X)$, we deduce that $\chi(G_8) = \chi_{\max}(X)$ and $\chi(G_9) = \chi'_{\max}(X)$, so $X$ admits a maximal ascending (resp.\ descending) $k$-colouring if an only if $G_8$ (resp.\ $G_9$) is $k$-colourable. Furthermore, both $G_8$ and $G_9$ are connected as a consequence of Proposition \ref{prop:grconexos}. Therefore, for $i = 8,9$ it suffices to define $\mathcal{M}_k^i(X) = S_k(G_i)$ and apply Theorem \ref{theorem:lechugamurillo}.

	We now consider the minimal colourings. It follows from Definition \ref{definition:maximalminimalcolourings} that $\chi_{\min}(X) = \min_{0\le r < \dim X}\{\chi_r(X)\}$ and that $\chi'_{\min}(X) = \min_{0 < s \le \dim X}\{\chi'_r(X)\}$. By a reasoning analogous to that of Proposition \ref{proposition:kscolouring}, the desired algebras are
	\[\mathcal{M}_k^{10}(X) = S_k\big(G_0(X)\big) \otimes S_k\big(G_1(X)\big)\otimes \dots \otimes S_k\big(G_{n-1}(X)\big),\]
	\[\mathcal{M}_k^{11}(X) = S_k\big(G'_1(X)\big) \otimes S_k\big(G'_2(X)\big)\otimes \dots \otimes S_k\big(G'_{n}(X)\big).\]
	The result follows.
	\end{proof}

Theorem \ref{theorem:all} now follows immediately from Propositions \ref{proposition:hyperedgeColourings}, \ref{proposition:simplicialColourings}, \ref{proposition:kscolouring} and \ref{proposition:maximalMinimal}.

\section{Algorithmic complexity of simplicial complex colourings}

If $G$ is a graph, it can be regarded as a simplicial complex $X(G)$ whose $0$-simplices and $1$-simplices are, respectively, the vertices and edges of $G$. Such a simplicial complex can be encoded using an adjacency matrix, so its codification has the same length as that of $G$. 

In this section we show that the (edge, total) colourability of a graph $G$ is equivalent to the $\mathfrak{C}_i$-colourability of $X(G)$ for certain indices $i$. As a consequence, we immediately deduce Theorem \ref{theorem:NP-hardness}.

\begin{remark}\label{remark:simpleNP}
	It is immediate that the $k$-colourability of a graph $G$ is equivalent both to the $\mathfrak{C}_1$-$k$-colourability and the $\mathfrak{C}_7$-$(k,1)$-colourability of $X(G)$. Similarly, the total $k$-colourability of $G$ is equivalent to the $\mathfrak{C}_6$-$k$-colourability of $X(G)$. 
\end{remark}

\begin{proposition}\label{proposition:vertexNP}
	The $k$-colourability of a graph $G$ is equivalent both to the $\mathfrak{C}_4$-$(k+1)$-colourability and the $\mathfrak{C}_i$-$k$-colourability, $i = 8,10$, of $X = X(G)$.
\end{proposition}

\begin{proof}
	We begin with the $\mathfrak{C}_4$-colourability. Let $\psi\colon V\to \{1,2,\dots,k\}$ be a $k$-colouring of $G$. Then, the map $\varphi\colon X\to \{1,2,\dots,k+1\}$ defined by
		\[\varphi(\sigma) = \begin{cases}
			\psi(\sigma), & \text{if $\sigma\in X^0$}, \\
			k+1, & \text{if $\sigma \in X^1$}.
		\end{cases}\]
	is a $\mathfrak{C}_4$-$(k+1)$-colouring of $X$. Reciprocally, if $\varphi\colon X\to \{1,2,\dots,k+1\}$ is a $\mathfrak{C}_4$-$(k+1)$-colouring of $X$, we may assume that at least one $1$-simplex receives image $k+1$, so $k+1\not\in\varphi(X^0)$. The map $\psi\colon X^0 = V\to\{1,2,\dots,k\}$ taking $v$ to $\psi(v) = \varphi(\{v\})$ is a $k$-colouring of $G$.

	We now consider the $\mathfrak{C}_i$-$k$-colourability, $i = 8,10$. First, if $\psi\colon V \to \{1,2,\dots,k\}$ is a $k$-colouring of $G$, $X$ admits a $\mathfrak{C}_i$-$k$-colouring defined by
	\[
		\varphi(\sigma) = \begin{cases}
		\psi(\sigma), & \text{if $\sigma \in X^0$}, \\
		k, & \text{if $\sigma \in X^1$}.
		\end{cases}
	\]
	Reciprocally, if $\varphi\colon X\to \{1,2,\dots,k\}$ is a $\mathfrak{C}_i$-$k$-colouring of $X$, $i = 8,10$, the restriction $\psi =\varphi_{|X^0} \colon X^0= V \to \{1,2,\dots,k\}$ is a $k$-colouring of $G$.
\end{proof}

\begin{proposition}\label{proposition:edgeNP}
	The edge $k$-colourability of a graph $G$ is equivalent both to the $\mathfrak{C}_5$-$(k+1)$-colourability and the $\mathfrak{C}_i$-$k$-colourability, $i=9,11$, of $X=X(G)$.
\end{proposition}

\begin{proof}
	We begin with the $\mathfrak{C}_5$-$(k+1)$-colourability. If $\psi\colon E\to \{1,2,\dots,k\}$ is an edge $k$-colouring of $G$, the map $\varphi\colon X(G) = X\to \{1,2,\dots,k+1\}$ defined by
		\[\varphi(\sigma)=\begin{cases}
		\psi(\sigma), & \text{if $\sigma \in X^1$}, \\
		k+1, & \text{if $\sigma \in X^0$}.
		\end{cases}\]
	is a $\mathfrak{C}_5$-$(k+1)$-colouring of $X$. Reciprocally, if $\varphi\colon X\to \{1,2,\dots,k+1\}$ is a $\mathfrak{C}_5$-$(k+1)$-colouring of $X$, we may suppose that at least one 0-simplex receives image $k+1$, thus $k+1\not\in\varphi(X^1)$. Then, the map $\psi = \varphi_{|X^1}\colon X^1 = E\to \{1,2,\dots,k\}$ is an edge $k$-colouring of $G$.

	We continue with the $\mathfrak{C}_i$-$k$-colourability, $i=9,11$. If $\psi\colon E\to \{1,2,\dots,k\}$ is an edge $k$-colouring of $G$, $X$ admits a $\mathfrak{C}_i$-$k$-colouring $\varphi \colon X\to \{1,2,\dots,k\}$ defined by
		\[\varphi(\sigma)= \begin{cases}
		\psi(\sigma), & \text{if $\sigma \in X^1$},\\
		k, & \text{if $\sigma \in X^0$}.
		\end{cases}\]
	Reciprocally, if $\varphi\colon X\to \{1,2,\dots,k\}$ is a $\mathfrak{C}_i$-$k$-colouring of $X$, $i = 9,11$, the restriction $\psi =\varphi_{|X^1} \colon X^1= E \to \{1,2,\dots,k\}$ is an edge $k$-colouring of $G$.
\end{proof}

Finally, Theorem \ref{theorem:NP-hardness} follows immediately from Remark \ref{remark:simpleNP}, Proposition \ref{proposition:vertexNP}, Proposition \ref{proposition:edgeNP} and the algorithmic complexity of the problem of (edge, total) $k$-colourability of graphs.

\section*{Acknowledgements}
The author is thankful to J.\ M.\ Carball\'es, C.\ Costoya and A.\ Viruel for their insightful comments and guidance. The author is also thankful to the referee, whose comments helped improve the presentation of this work.

\end{document}